\documentclass[12pt,leqno]{amsart}
\usepackage{amsmath,amsthm,amssymb,amsfonts}
\usepackage{eucal,enumerate}
\usepackage{xypic,graphicx,psfrag}
\usepackage{mathrsfs}
\usepackage{hyperref}

\usepackage{color} 

\usepackage{tikz}
\usetikzlibrary{decorations.pathreplacing}
\usepackage{scalefnt}

\newcommand\sectionpage\newpage
\renewcommand\sectionpage{}

\newtheorem{lem}{Lemma}[section]
\newtheorem{cor}[lem]{Corollary}

\newtheorem{thm}[lem]{Theorem}

\theoremstyle{definition}
\newtheorem{conj}[lem]{Conjecture}

\numberwithin{equation}{section}
\numberwithin{table}{section}
\numberwithin{figure}{section}

\allowdisplaybreaks

\newcommand\vstrut[1]{\rule{0ex}{#1}}

\renewcommand\mod{\, \operatorname{mod}\, }
\renewcommand{\phi}{\varphi} 
\renewcommand{\epsilon}{\varepsilon}
\newcommand\eset{\varnothing}
\newcommand\bs{\ensuremath{\backslash}}
\newcommand\codim{\operatorname{codim}}

\newcommand\cA{\mathscr{A}}		

\renewcommand\cH{\mathcal{H}}	
\renewcommand\cL{\mathscr{L}}	
\newcommand\cP{\mathcal{P}}
\newcommand\cU{\mathcal{U}}	
\newcommand\cW{\mathcal{W}}
\newcommand\bbR{\mathbb{R}}
\newcommand\bbZ{\mathbb{Z}}
\newcommand\pP{\mathbb P}	
\newcommand\pQ{\mathbb Q}	

\newcommand\bz{\mathbf z}
\newcommand\bL{\mathbf{L}}

\newcommand\hatc{{\hat c}}
\newcommand\hatd{{\hat d}}

\newcommand\vol{\operatorname{vol}}
\newcommand\M{\mathbf{M}}
\newcommand\Kot{Kot\v{e}\v{s}ovec}
\newcommand\cube{[0,1]^{2q}}
\newcommand\ocube{(0,1)^{2q}}

\allowdisplaybreaks

\newcommand\Eiopmu{(I.2.1)}		

\newcommand\background{I.2}	
\newcommand\types{I.5}		
\newcommand\subspaces{I.3.2}	
\newcommand\oneortwo{II.3}		
\newcommand\R{IV.6}		

\newcommand\LWmu{I.3.1}		
\newcommand\Ttypenumber{I.5.3}	
\newcommand\Ptwopiecetypes{I.5.6}	
\newcommand\pattackdc{II.3.1}	
\newcommand\TutwoP{II.3.1}		
\newcommand\Tgammapolysquare{II.5.1}	
\newcommand\thmN{IV.9.2}		

\begin{document}

\allowdisplaybreaks

\title{A $q$-Queens Problem \\
III.  Nonattacking Partial Queens}

\author{Seth Chaiken}
\address{Computer Science Department\\ The University at Albany (SUNY)\\ Albany, NY 12222, U.S.A.}
\email{\tt sdc@cs.albany.edu}

\author{Christopher R.\ H.\ Hanusa}
\address{Department of Mathematics \\ Queens College (CUNY) \\ 65-30 Kissena Blvd. \\ Queens, NY 11367-1597, U.S.A.}
\email{\tt chanusa@qc.cuny.edu}

\author{Thomas Zaslavsky}
\address{Department of Mathematical Sciences\\ Binghamton University (SUNY)\\ Binghamton, NY 13902-6000, U.S.A.}
\email{\tt zaslav@math.binghamton.edu}

\begin{abstract}
We apply our geometrical theory for counting placements of $q$ nonattacking on an $n\times n$ chessboard, from Parts~I and II, to partial queens: that is, chess pieces with any combination of horizontal, vertical, and $45^\circ$-diagonal moves.  Parts~I and II showed that for any rider (a piece with moves of unlimited length) the answer will be a quasipolynomial function of $n$ in which the coefficients are essentially polynomials in $q$.  Those general results gave the three highest-order coefficients of the counting quasipolynomial and formulas for counting placements of two nonattacking pieces and the combinatorially distinct types of such placements. 

By contrast, the unified framework we present here for partial queens allows us to explicitly compute the four highest-order coefficients of the counting quasipolynomial, show that the five highest-order coefficients are constant (independent of $n$), and find the period of the next coefficient (which depends upon the exact set of moves).  Furthermore, for three nonattacking partial queens we are able to prove formulas for the total number of nonattacking placements and for the number of their combinatorially distinct types.  

The method of proof, as in the previous parts, is by detailed analysis of the lattice of subspaces of an inside-out polytope.
\end{abstract}



\thanks{Version of \today.}

\thanks{The outer authors thank the very hospitable Isaac Newton Institute for facilitating their work on this project. The inner author gratefully acknowledges support from PSC-CUNY Research Awards PSCOOC-40-124, PSCREG-41-303, TRADA-42-115, TRADA-43-127, and TRADA-44-168.}

\maketitle
\pagestyle{myheadings}
\markright{\textsc{A $q$-Queens Problem. III.  Nonattacking Partial Queens}}\markleft{\textsc{Chaiken, Hanusa, and Zaslavsky}}

\vspace{-.2in}

\tableofcontents

\section{Introduction}\label{intro}

The well known $n$-Queens Problem asks for the number of ways to place $n$ nonattacking queens on an $n\times n$\label{d:n} square chessboard.  A broader question separates the number of pieces from the size of the board; that question is the $q$-Queens Problem, which asks for the number of ways to place $q$\label{d:q} nonattacking queens on an $n\times n$ board.  This article is part of a series in which we develop a general method for solving such questions for pieces of the type called ``riders'', whose moves have unlimited distance \cite{QQs1, QQs2, QQs4, QQs5, QQs6}.  

In Part~I we obtained a general form for the function of $n$ that, given $q$ identical riders, counts the number of nonattacking configurations of those pieces on an $n\times n$ board.
In Part~II we learned that the complexity of that counting function depends on the magnitudes of the numerator and denominator of the slopes of the piece's move directions.  Thus it is natural to focus on ``partial queens''---the rider pieces whose moves are a subset of those of the queen---because they are the pieces for which the move slopes involve only $\pm1$ and $0$.  By narrowing our focus to partial queens we are able to ascertain much more about the counting functions, and in a unified manner.

We proved in Part~I that in each non-attacking placement problem the number of solutions is a quasipolynomial function of $n$---that means it is given by a cyclically repeating sequence of polynomials as $n$ varies---and that the coefficient of each power of $n$ is (up to a factor) a polynomial function of $q$.  In Part~II we found, for instance, that the coefficients of the three highest powers of $n$ do not vary with $n$.  
In the main theorem here, Theorem~\ref{P:hvdiag}, we are able to say much more for partial queens.  Most importantly, we prove explicit formulas for the coefficients of the four highest powers of $n$ and for the leading term (in powers of $q$) of the coefficient of every power of $n$.  Surprisingly, we are even able to obtain formulas for the periodic parts of the coefficients of the next two highest powers of $n$, though not for the nonperiodic parts.  A consequence is that we know something about the period of the quasipolynomial; in particular, if a partial queen does not have both diagonal moves, then all the highest six coefficients are constant.  (That is definitely not true for the bishop or queen---Part~VI proves that the period of the counting quasipolynomial for $q\geq 3$ bishops is 2.)
Furthermore, although the only fully explicit counting function we could find for $q$ arbitrary riders on the square board was for $q=2$ (Theorem~\TutwoP), for partial queens we get the complete the counting function for $3$ nonattacking pieces (Theorem~\ref{T:3Qhk}).

The method of proof is purely an application of the theory from Parts~I and II (whose essentials are reviewed in Section~\ref{essentials}).  The chess problem is converted into a geometry problem in which moves become hyperplanes in $\bbR^{2q}$; the $n\times n$ board becomes the set of $1/(n+1)$-fractional lattice points inside the unit square; and the number of nonattacking configurations becomes a linear combination of the numbers of $q$-tuples of these lattice points that lie in intersections of move hyperplanes.  We explicitly determine all subspaces of codimension 1, 2, and 3 in the lattice of such intersections and count the number of lattice points therein (from which follows the count for $q=3$ pieces).   We further apply our theory to calculate the number of combinatorially distinct configurations of three nonattacking partial queens, which turns out to be determined solely by the number of moves, not which moves they are (Corollary~\ref{C:types23}).  

We mentioned the relative simplicity of partial queens.  A deeper reason we study this set of pieces is that we hope ultimately to discover the factors that control such basic properties of the counting formula as the period of the cyclically repeating polynomials, the periods of the individual coefficients of powers of $n$, formulas for the coefficients in terms of the moves of the piece under consideration, or anything that will let us predict aspects of the counting functions by knowing the moves. 
For this hope, partial queens can be a valuable test set, not as hard as general riders but varied enough to suggest patterns for counting functions---indeed, it was the formulas and their proofs for partial queens that led us to several of the general properties proved in Parts~I and II.  

Our analysis involves a great deal of notation; we append a notational dictionary, which follows an observation and a question in Section~\ref{last} and the detailed subspace analysis in Section~\ref{proofs}.

\sectionpage\section{Essentials}\label{essentials}

\subsection{Review}\

We assume acquaintance with the notation and methods of Parts~I and II as they apply to the square board.   
For easy reference we review the most important here.

The square board consists of the integral points in the interior of the integral multiple $(n+1) [0,1]^2$\label{d:n+1} of the unit square.  
Writing $[n] := \{1,\ldots,n\}$\label{d:[n]}, the set of points of the board is
$$
[n]^2 = (n+1)(0,1)^2 \cap \bbZ^2.
\label{d:[n]2}
$$ 

We write $\delta_{ij}$\label{d:KD} for the Kronecker delta.

A \emph{move} of a piece $\pP$\label{d:P} is the difference between two positions on the board; it may be any integral multiple of a vector in a finite, nonempty set $\M$\label{d:moveset} of \emph{basic moves}.  The latter are non-zero, non-parallel integral vectors $m_r = (c_r,d_r)$\label{d:mr} in lowest terms, i.e., $c_r$ and $d_r$ are relatively prime.  
(The slope $d_r/c_r$ contains all  necessary information and can be specified instead of $m_r$ itself.)  
One piece \emph{attacks} another if the former can reach the latter by a move.  The constraint is that no two pieces may attack one another, or to say it mathematically, if there are pieces at positions $z_i$ and $z_j$, then $z_j-z_i$ is not a multiple of any $m_r$.  
For a move $m=(c,d)$, we define 
$$
\hatc:=\min(|c|,|d|), \quad \hatd:=\max(|c|,|d|).
\label{d:cdhat}
$$

We assume that $q>0$.  
We treat configurations of $q$ pieces as $1/(n+1)$-fractional lattice points in the $2q$-dimen\-sion\-al inside-out polytope $(\ocube,\cA_\pP)$,\label{d:cP} where $\cA_\pP$ is the \emph{move arrangement}\label{d:AP} whose members are the \emph{move hyperplanes} (or \emph{attack hyperplanes}) 
$$
\cH^{d/c}_{ij} := \{ \bz \in \bbR^{2q} : (z_j - z_i) \cdot (d,-c) = 0 \}.
\label{d:slope-hyp}
$$ 
(Inside-out polytopes are explained in Section~\background.)  The equation of a move hyperplane is called a \emph{move equation} or \emph{attack equation}.  
We view a coordinate vector $\bz\in\bbR^{2q}$ as consisting of $q$ planar vectors, $\bz=(z_1,z_2,\ldots,z_q)$,\label{d:config} where $z_i=(x_i,y_i)\in \bbR^2$.  
The \emph{intersection lattice} $\cL(\cA_\pP)$\label{d:L} is the lattice of all intersections of subsets of the move arrangement, ordered by reverse inclusion; its M\"obius function is $\mu$\label{d:mu}.  
Of the $1/(n+1)$-fractional points in $\ocube$, those in move hyperplanes represent attacking configurations, which we exclude by M\"obius inversion over the intersection lattice; the others represent nonattacking configurations.  
The number of nonattacking configurations of $q$ unlabelled pieces on an $n\times n$ board is $u_\pP(q;n)$\label{d:indistattacks}, whose full expression is 
\begin{equation*}
u_\pP(q;n) = \gamma_0(n) n^{2q} + \gamma_1(n) n^{2q-1} + \gamma_2(n) n^{2q-2} + \cdots + \gamma_{2q}(n) n^0.
\label{d:gamma}
\end{equation*}
What we actually compute is the number of nonattacking labelled configurations, $o_\pP(q;n)$,\label{d:distattacks} which equals $q!u_\pP(q;n)$.  
The Ehrhart theory of inside-out polytopes implies that these counting functions are quasipolynomials in $n$. 

In Part~II we defined $\alpha(\cU;n)$\label{d:alphaU} as the number of points in the intersection of the essential part of an intersection subspace $\cU\in\cL(\cA_\pP)$\label{d:U} with the integral hypercube $[n]^{2\kappa}$, where $\kappa$\label{d:kappa} is the number of pieces involved in the move equations defining $\cU$.  (The \emph{essential part} is the restriction of $\cU$ to the coordinate subspace of $\bbR^{2q}$ that involves only the coordinates of pieces that appear in those equations.)  By Ehrhart theory, $\alpha(\cU;n)$ is a quasipolynomial in $n$ of degree $2\kappa-\nu$ with constant leading coefficient, where $\nu$ is the codimension of $\cU$.  
The formula
\begin{equation}\label{E:iopmu}
q!u_{\pQ^{hk}}(q;n) = \sum_{\cU\in\cL(\cA_\pP)} \mu(\hat0,\cU) \alpha(\cU;n) n^{2q-2\kappa},
\end{equation}
from Equation~\Eiopmu\ with $t=n+1$ and $E_{\cU \cap \cP^\circ}(t) = \alpha(\cU;n) n^{2q-2\kappa}$, is the foundation stone of this paper.
We also defined the abbreviations 
$$
\alpha^{d/c}(n) := \alpha(\cH_{12}^{d/c};n),
\label{d:adc}
$$  
the number of ordered pairs of positions that attack each other along slope $d/c$ (they may occupy the same position; that is considered attacking).  Similarly,  
$$
\beta^{d/c}(n) := \alpha(\cW_{123}^{\,d/c};n),
\label{d:bdc}
$$
the number of ordered triples that are collinear along slope $d/c$; $\cW_{123}^{\,d/c} := \cH_{12}^{d/c} \cap \cH_{23}^{d/c}$.\label{d:Wdc}
Proposition~\pattackdc\ gives general formulas for $\alpha$ and $\beta$.
We need only a few examples in Part~III:
\begin{equation}
\begin{aligned}
\alpha^{0/1}(n) = \alpha^{1/0}(n) &= n^3, \qquad
\alpha^{\pm1/1}(n) = \frac{2n^3+n}{3}, 
\\
\beta^{0/1}(n) = \beta^{1/0}(n) &= n^4, \qquad
\beta^{\pm1/1}(n) = \frac{n^4+n^2}{2} .
\end{aligned}
\label{E:attacklines}
\end{equation}

\subsection{Partial queens}\

A \emph{partial queen} is a piece $\pQ^{hk}$,\label{d:partQ} whose moves are $h$ horizontal and vertical moves and $k$ diagonal moves of slopes $\pm1$, where $h, k \in \{0,1,2\}$ and (to avoid the trivial case $\M=\eset$) we assume $h+k\geq1$.  This includes the cases of the bishop ($h=0$ and $k=2$) and the queen ($h=k=2$), and allows for pieces such as the \emph{semiqueen} ($h=2$ and $k=1$) and the \emph{anassa} ($h=k=1$).\footnote{\Kot\ calls our anassa ``semi-rook + semi-bishop'' but we want it to have a distinctive name.  ``Anassa'' is archaic Greek feminine for a tribal chief, i.e., presumably for the consort of a chief \cite{Anax}.}  By restricting to partial queens it is possible to explicitly calculate the contributions to $q!u_{\pQ^{hk}}(q;n)$ of intersection subspaces up to codimension 3.  
From this, we can calculate the coefficients $\gamma_1$, $\gamma_2$, and $\gamma_3$ and the counting quasipolynomials $u_{\pQ^{hk}}(2;n)$ and $u_{\pQ^{hk}}(3;n)$.

\sectionpage\section{Coefficients}\label{coeffs}

\Kot\ proposed formulas for the coefficients $\gamma_1$ and $\gamma_2$ of the counting quasipolynomials for queens and bishops and other riders~\cite[third ed., pp.\ 13, 210, 223, 249, 652, 663; also in later eds.]{ChMath}.  Our main theorem proves the generalization of his conjectures to partial queens and to $\gamma_3$: our formulas for $\gamma_3$ for the queen $\pQ^{22}$, the anassa $\pQ^{11}$, the semiqueen $\pQ^{21}$, and the \emph{trident} $\pQ^{12}$ are new.  
(They have been supported in small cases by calculations by \Kot; in fact, we used his collection of formulas for the queen and bishop to help correct errors in algebra.) 

\begin{thm}\label{P:hvdiag}
{\rm(I)}  For a partial queen $\pQ^{hk}$, the coefficient $q!\gamma_i$ of $n^{2q-i}$ in $o_{\pQ^{hk}}(q;n)$ is a polynomial in $q$, periodic in $n$, with leading term 
$$
\left(-\,\frac{3h+2k}{6}\right)^i \frac{q^{2i}}{i!}.
$$

{\rm(II)}  The coefficients $\gamma_0,\ldots,\gamma_4$ of the five highest powers of $n$ in the quasipolynomial $u_{\pQ^{hk}}(q;n)$ are independent of $n$.

The coefficients $\gamma_i$ for $i=1,2,3$ are given by
\begin{equation}\label{eq:gamma2q-1}
\gamma_1 = -\, \frac{1}{(q-2)!} \bigg\{ \frac{3h+2k}{6} \bigg\} ,
\end{equation}
\begin{equation}
\begin{aligned}
\gamma_2 
&= \frac{1}{2!(q-2)!} \bigg\{ (q-2)_2 \Big(\frac{3h+2k}{6}\Big)^2 + (q-2) \frac{4h+2k+8hk+12\delta_{h2}+5\delta_{k2}}{6} + (h+k-1) \bigg\} ,
\end{aligned}
\label{eq:gamma2q-2}
\end{equation}
and
\begin{equation}
\begin{aligned}
\gamma_3 = 
-\,\frac{1}{3!(q-2)!} \bigg\{& 
(q-2)_4\Big(\frac{3h+2k}{6}\Big)^3 
\\& 
+ (q-2)_3\frac{(3h+2k)(4 h+8 h k+2 k+12  \delta_{h2}+5  \delta_{k2})}{12} 
\\&
+(q-2)_2\frac{30 h^2+257 h k+20 k^2-8 k+40 (8 k+9)  \delta_{h2}+4 (51 h+26)  \delta_{k2}}{20}
\\&
+ (q-2)\frac{12h(h-1)+20 h k+8k(k-1)+8k \delta_{h2}+5 h\delta_{k2}}{2}
\\&
+ k 
\bigg\}.
\end{aligned}
\label{eq:gamma2q-3}
\end{equation}

{\rm(III)}  The next coefficient, $\gamma_5$, is constant except that it has period $2$ if $k=2$ and $h\neq0$ (and $q\geq3$), with periodic part $- (-1)^n h/8(q-3)!$.  
\end{thm}

We write the falling factorials in terms of $q-2$ instead of $q$ because every nontrivial coefficient $\gamma_i$ ($\gamma_0=1/q!$ being ``trivial'') has a numerator factor $(q)_j$ with $j\geq2$ and a denominator factor $q!$ (since $u_\pP = o_\pP/q!$).  Therefore $u_\pP(q;n)$ as a whole looks like 
$$\frac{n^{2q}}{q!} + \frac{(q)_2\text{(nontrivial quasipolynomial in $n$ and $q$)}}{q!}.$$  
It seems natural to cancel the repetitious factor $(q)_2$ in every coefficient other than $\gamma_0$.  

Curiously, $\delta_{h2}=h(h-1)/2$ and $\delta_{k2}=k(k-1)/2$ because $h,k\in\{0,1,2\}$.  Thus, the expressions involving these Kronecker deltas can be written as polynomials in $h$ and $k$.  We do not see a reason to prefer one form over the other.

Tables~\ref{Tb:pqueensgamma2}--\ref{Tb:pqueensgamma3} give the explicit formulas for the coefficients $\gamma_2$ and $\gamma_3$ for the partial queens.

\begin{table}[ht]
\begin{tabular}{|l|l|l|} \hline
\quad \emph{Name} & $(h,k)$ \vstrut{15pt} &\quad $\gamma_2$ 
\\[4pt] \hline
Semi-rook & $(1,0)$ 
\vstrut{22pt} 
& $\displaystyle\frac{1}{2!(q-2)!}  \Big\{  \Big(\frac{1}{2}\Big)^2(q-2)_2+\frac{2}{3}(q-2)\Big\}$
\\[10 pt] 
Rook & $(2,0)$ 
&$\displaystyle\frac{1}{2!(q-2)!}  \Big\{  (1)^2(q-2)_2+\frac{10}{3}(q-2)+1\Big\}$
\\[10 pt] 
Semibishop & $(0,1)$ 
& $\displaystyle\frac{1}{2!(q-2)!}  \Big\{  \Big(\frac{1}{3}\Big)^2(q-2)_2+\frac{1}{3}(q-2)\Big\} $
\\[10 pt] 
Anassa & $(1,1)$ 
& $\displaystyle\frac{1}{2!(q-2)!} \Big\{ \Big(\frac{5}{6}\Big)^2(q-2)_2  +  \frac{7}{3} (q-2)+ 1 \Big\} $
\\[10 pt] 
Semiqueen & $(2,1)$ 
& $\displaystyle\frac{1}{2!(q-2)!} \Big\{ \Big(\frac{4}{3}\Big)^2(q-2)_2  +  \frac{38}{6} (q-2)+ 2 \Big\}$
\\[10 pt] 
Bishop & $(0,2)$ 
& $\displaystyle\frac{1}{2!(q-2)!} \Big\{  \Big(\frac{2}{3}\Big)^2 (q-2)_2 + 3 (q-2) + 1 \Big\}$
\\[10 pt] 
Trident & $(1,2)$ 
& $\displaystyle\frac{1}{2!(q-2)!} \Big\{  \Big(\frac{7}{6}\Big)^2(q-2)_2 + \frac{29}{6}(q-2)  + 2 \Big\}$
\\[10 pt] 
Queen & $(2,2)$ 
& $\displaystyle\frac{1}{2!(q-2)!} \Big\{  \Big(\frac{5}{3}\Big)^2 (q-2)_2 + \frac{61}{6} (q-2) + 3 \Big\} $
\\ [12 pt] \hline
\end{tabular}
\bigskip
\caption{The coefficient $\gamma_2$ for the various partial queens.} 
\label{Tb:pqueensgamma2}
\end{table}
\begin{table}[ht]
\begin{tabular}{|l|l|} \hline
$(h,k)$ \vstrut{15pt} &\qquad $\gamma_3$ 
\\[4pt] \hline
$(1,0)$ 
\vstrut{22pt} 
&$\displaystyle- \frac{1}{3!(q-2)!} \Big\{ \Big(\frac{1}{2}\Big)^3(q-2)_4 + (q-2)_3 + \frac{3}{2}(q-2)_2  \Big\}$
\\[10 pt] 
$(2,0)$ 
&$\displaystyle- \frac{1}{3!(q-2)!} \Big\{ (1)^3(q-2)_4 + 10(q-2)_3 + 24(q-2)_2 + 12(q-2) \Big\}$
\\[10 pt] 
$(0,1)$ 
& $\displaystyle- \frac{1}{3!(q-2)!} \Big\{ \Big(\frac{1}{3}\Big)^3(q-2)_4 + \frac{1}{3}(q-2)_3 + \frac{3}{5}(q-2)_2 + 1 \Big\}$
\\[10 pt] 
$(1,1)$ 
& $\displaystyle- \frac{1}{3!(q-2)!} \Big\{ 
\Big(\frac{5}{6}\Big)^3(q-2)_4 + \frac{35}{6}(q-2)_3 +\frac{299}{20} (q-2)_2+ 10(q-2) + 1 \Big\}$
\\[10 pt] 
$(2,1)$ 
& $\displaystyle- \frac{1}{3!(q-2)!} \Big\{ 
\Big(\frac{4}{3}\Big)^3(q-2)_4 + \frac{76}{3}(q-2)_3 +\frac{663}{10}(q-2)_2 + 36(q-2) + 1 
\Big\}$
\\[10 pt] 
$(0,2)$ 
& $\displaystyle-\,\frac{1}{3!(q-2)!} \Big\{ 
\Big(\frac{2}{3}\Big)^3 (q-2)_4 + 3(q-2)_3 + 10(q-2)_2 + 8(q-2) + 2 
\Big\}$
\\[10 pt] 
$(1,2)$ 
& $\displaystyle- \frac{1}{3!(q-2)!} \Big\{ 
\Big(\frac{7}{6}\Big)^3(q-2)_4 + \frac{203}{12}(q-2)_3 +\frac{458}{10}(q-2)_2 + \frac{61}{2}(q-2) + 2 
\Big\}$
\\[10 pt] 
$(2,2)$ 
& $\displaystyle-\, \frac{1}{3!(q-2)!} \Big\{ \Big(\frac{5}{3}\Big)^3 (q-2)_4 + \frac{305}{6} (q-2)_3 + \frac{681}{5} (q-2)_2 + 73 (q-2) + 2 \Big\} $
\\ [12 pt] \hline
\end{tabular}
\bigskip
\caption{The coefficient $\gamma_3$ for the various partial queens.} 
\label{Tb:pqueensgamma3}
\end{table}
%


Theorem~\ref{P:hvdiag} yields a nice corollary for pieces with only one diagonal move (or none, but those pieces, the rook and half-rook of Section~\R, are elementary).

\begin{cor}\label{C:Qh1}
When the piece is a partial queen $\pQ^{hk}$ with $k<2$, the six leading coefficients, $\gamma_i$ for $i\leq5$, are independent of $n$.
\end{cor}

\begin{proof}[{Proof of Theorem~\ref{P:hvdiag}}] \ 
Theorem~\Tgammapolysquare\ says that $(q)_{2i}$ gives the highest power of $q$ and its coefficient is $(-a_{10}/2)^i/i!$, where $a_{10} = \sum_{(c,d)\in\M} (3\hatd-\hatc)/3\hatd^2 = h\frac{3}{3}+k\frac{2}{3}$ since there are $h$ moves with $(\hatc,\hatd)=(0,1)$ and $k$ with $(\hatc,\hatd)=(1,1)$.

The coefficient $\gamma_1$ is from Theorem~\Tgammapolysquare.  For the other coefficients we prove two lemmas that state the total contributions to $u_{\pQ^{hk}}(q;n)$ from subspaces of all codimensions $\nu\leq3$, the proof of which, involving case-by-case analysis, we postpone to Section~\ref{proofs}.

\begin{lem}\label{L:codim012}
The contributions to $u_{\pQ^{hk}}(q;n)$ from subspaces of codimension $\nu\leq2$ are as follows.

{\rm(I)}  From $\codim\cU=0${\rm:}  
\begin{equation}
\frac{1}{q!} n^{2q}.
\label{E:codim0}
\end{equation}

{\rm(II)}  From $\codim\cU=1${\rm:}  
\begin{equation}
-\frac{1}{q!} \left\{ (q)_2 \frac{3h+2k}{6} n^{2q-1} + (q)_2 \frac{k}{6} n^{2q-3} \right\}.
\label{E:codim1}
\end{equation}

{\rm(III)}  From $\codim\cU=2${\rm:}  
\begin{equation}
\begin{aligned}
\frac{1}{q!} \bigg\{ &\bigg[ (q)_4 \frac12 \Big(\frac{3h+2k}{6}\Big)^2 + (q)_3 \frac{4h+2k+8hk+12\delta_{h2}+5\delta_{k2}}{12} + (q)_2 \frac{h+k-1}{2} \bigg] n^{2q-2} \\
&+ \bigg[ (q)_4 \frac{k(3h+2k)}{36} + (q)_3 \frac{k(2h+1)+2\delta_{k2}}{6} \bigg] n^{2q-4} \\
&+ \bigg[ (q)_4 \frac{k^2}{72} + (q)_3 \big[1-(-1)^n\big] \frac{\delta_{k2}}{8} \bigg] n^{2q-6} \bigg\}.
\end{aligned}
\label{E:codim2}
\end{equation}
\end{lem}

\begin{lem}\label{L:codim3}
The total contribution to $u_{\pQ^{hk}}(q;n) = \frac{1}{q!} o_\pP(q;n)$ from subspaces of codimension $3$ is
\begin{align*}
-\frac{1}{q!} &\bigg\{ 
\bigg[ \ 
n^{2q-3} \bigg( 
(q)_3\frac{12h(h-1)+20hk+8k(k-1)+8 k \delta_{h2}+5 h\delta_{k2}}{12}
\\&\qquad\qquad 
+(q)_4\frac{30 h^2+257 h k+20 k^2-8 k+40 (8 k+9) \delta_{h2}+ 4(51 h+26) \delta_{k2}}{120}
\\&\qquad\qquad 
+(q)_5\frac{(3h+2k)(4 h+8 h k+2 k+12 \delta_{h2}+5 \delta_{k2})}{72} 
+(q)_6\frac{(3h+2k)^3}{1296} 
\bigg) 
\\&\quad 
+ n^{2q-5} \bigg( 
(q)_3\frac{8 k(h+k-1)+8 k\delta_{h2}+11 h \delta_{k2}}{24}
\\&\qquad\qquad 
+(q)_4\frac{k(31 h+2k+2)+32k \delta_{h2}+(34 h+24) \delta_{k2}}{24}
\\&\qquad\qquad 
+(q)_5\frac{2 k \left(6 h^2+8 h k+5 h+3 k\right)+12 k  \delta_{h2}+(12 h+13 k) \delta_{k2}}{72}
\\&\qquad\qquad 
+(q)_6\frac{k (3h+2k)^2}{432} 
\bigg) 
\\&\quad 
+ n^{2q-7} \bigg( 
(q)_4\frac{2k(4h-1)+(61h+76) \delta_{k2}}{120}\\
&\qquad\qquad 
+(q)_5\frac{4k^2(2h+1)+(9h+14k)\delta_{k2}}{144}
+(q)_6\frac{k^2 (3h+2k)}{432} \bigg)
\\&\quad 
+ n^{2q-9} \bigg( 
(q)_5 \frac{k \delta_{k2}}{48} 
+(q)_6 \frac{k^3}{1296} 
\bigg) \bigg] \\&
- (-1)^n \delta_{k2} \bigg[ 
n^{2q-5} (q)_3 \frac{h}{8} + n^{2q-7} \bigg( (q)_4\frac{3h+4}{8} + (q)_5\frac{3h+2k}{48} 
\bigg) 
+ n^{2q-9} (q)_5 \frac{k}{48} 
\bigg]
\bigg\}.
\end{align*}
\end{lem}

The lemmas show that the contribution from codimension $\nu$ involves only powers $n^{2q-i}$ for which $i\geq\nu$ and $i$ has the same parity as $\nu$.  We cannot fully explain this parity remark; Theorem II.4.2 does say there is no contribution to the coefficient of $n^{2q-\nu-1}$, but it says nothing about lower powers.  Ehrhart theory says that the leading coefficient is constant; thus a periodic part can only appear at $n^{2q-\nu-2}$.

The proof of these lemmas involve totaling the contributions to $q! u_{\pQ^{hk}}(q;n) = o_{\pQ^{hk}}(q;n)$ in Equation~\eqref{E:iopmu} from all subspaces of codimension $\nu$.  To do that we break down those subspaces into types.  We use notation of the form $\cU_{\kappa}^\nu$ or $\cU_{\kappa\mathrm{a}}^\nu$ to represent a subspace of codimension $\nu$\label{d:codim} in the intersection semilattice $\cL(\cA)$ that involves $\kappa$ pieces, with a letter index to differentiate between distinct types of subspace with these same numbers.  In addition, we wish to differentiate between those subspaces that are indecomposable and those that decompose into subspaces of smaller codimension; for the latter we write an asterisk after the number of pieces and we specify the exact constituent subspaces.  For example, we will have a subspace $\cU_{5^*\mathrm{a}}^3{:}\cU_2^1\cU_{3\mathrm{a}}^2$. 

For each type we determine  the M\"obius function $\mu(\hat0,\cU)$ and count the number of lattice points in the intersection $\cU\cap\ocube$.  To perform this count in type $\cU_{\kappa\mathrm{a}}^\nu$, we count the number of ways to place $\kappa$ attacking pieces in the designated way, and then multiply by $n^{2(q-\kappa)}$ for the number of ways to place the remaining pieces whose positions are not constrained.

\medskip
\noindent
{\em Continuation of proof of Theorem~\ref{P:hvdiag}.}
The subspaces $\cU$ that contribute to $\gamma_2$ are only those of codimension two, because no contributions to $\gamma_2$ come from subspaces of codimension $0$ or $1$.  To find $\gamma_2$ we extract the coefficient of $n^{2q-2}$ from \eqref{E:codim2}.  When calculating $\gamma_3$, there are contributions from the subspaces of codimensions 3 and 1.  
Combining their contributions implies that the coefficients $\gamma_3$ are as in \eqref{eq:gamma2q-3}.

The contribution to $\gamma_4$ from any subspace of codimension 3 is necessarily zero (again by Theorem II.4.2), and by our calculations above the contribution is constant for every subspace of lesser codimension.  Along with the constancy of the leading coefficient, this implies that $\gamma_4$ is constant for all partial queens.

A periodic contribution to $\gamma_5$ can arise only from subspaces of codimensions 1, 2, and 3, and by Lemma~\ref{L:codim012} only from codimension 3.  
The periodic parts of all codimension-3 subspaces are collected in Lemma~\ref{L:codim3}, in which the periodic coefficient of $n^{2q-5}$ is $- (-1)^n (q)_3 h \delta_{k2}/8,$ so that is the periodic part of $q!\gamma_5$.
If we hold $q$ fixed, the counting quasipolynomials for the queen and the trident are the only ones of partial queens that have non-constant coefficient $\gamma_5$, whose period is 2.  
\end{proof}

\sectionpage\section{Two and Three Partial Queens}\label{two-three}

These observations are on display when we use our theory to calculate the counting quasipolynomial $u_{\pQ^{hk}}(3;n)$.  The results agree with formulas proposed by \Kot, who supplemented his formulas for bishops and queens by independently calculating (but, as is his practice, not proving) the other cases in his fifth edition~\cite{ChMath} after we suggested studying partial queens.

Complete formulas for two or three partial queens are in Theorems~\ref{T:2Qhk} and \ref{T:3Qhk}.

\begin{thm}\label{T:2Qhk}
The counting quasipolynomial for two partial queens $\pQ^{hk}$ is
\begin{equation*}
u_{\pQ^{hk}}(2;n) = \frac{1}{2}n^4 -\frac{3h + 2k}{6} n^3 + \frac{h+k-1}{2}n^2 - \frac{k}{6} n.
\end{equation*}
\end{thm}

\begin{proof}
In Theorem~\TutwoP\ there are $h$ moves with $(\hatc,\hatd)=(0,1)$ and $k$ with $(\hatc,\hatd)=(1,1)$.  So, all $\hatd_r=1$ and $n\mod \hatd_r=0$.
\end{proof}

\begin{thm}\label{T:3Qhk}
The counting quasipolynomial for three partial queens $\pQ^{hk}$ is a polynomial when $k<2$ and has period $2$ when $k=2$.  The formula is 
\begin{equation*}
\begin{aligned}
u_{\pQ^{hk}}(3;n)
&=  
\frac{1}{6} n^6
- \frac{3h+2k}{6} n^5 
\\&\quad
+ \left[ \frac{5h+4hk+4k-3}{6} + \delta_{h2} + \frac{5\delta_{k2}}{12} \right] n^4 
\\&\quad
- \left[ \frac{ (h+k-1)(3h+2k) }{3} + \frac{k}{6} + \frac{2k\delta_{h2}}{3} +  \frac{5h\delta_{k2}}{12} \right] n^3 
\\&\quad
+ \left[ \frac{(h+k-1)^2(h+k+2)}{6} + \frac{(2h+1)k}{6} + \frac{\delta_{k2}}{3} \right] n^2
\\&\quad
- \left[ \frac{k(h+k-1)}{3} + \frac{k\delta_{h2}}{3} + \frac{11h\delta_{k2}}{24} \right] n 
+ \frac{\delta_{k2}}{8} 
\\&\quad
+ (-1)^n \frac{\delta_{k2}}{8} \big( hn - 1 \big).
\end{aligned}
\end{equation*}
\end{thm}

Tables~\ref{Tb:pqueenspolys2} and \ref{Tb:pqueenspolys3} list the quasipolynomials for the various partial queens.  
Theorem~\ref{T:2Qhk} gives $u_{\pQ^{00}}(2;n)=\binom{n^2}{2}$, $u_{\pQ^{10}}(2;n)=n^2\binom{n^2}{2}$, and $u_{\pQ^{20}}(2;n)=[(n)_2]^2$; 
Theorem~\ref{T:3Qhk} gives $u_{\pQ^{00}}(3;n)=\binom{n^2}{3}$, $u_{\pQ^{10}}(3;n)=n^3\binom{n^3}{3}$, and $u_{\pQ^{20}}(3;n)=[(n)_3]^2$; all as one expects from elementary counting (given that $\pQ^{00}$, the partial queen with no moves, attacks only a piece on the same square).
\begin{table}[ht]
\begin{tabular}{|l|l|l|} \hline
\quad \emph{Name} & $(h,k)$ \vstrut{15pt} &\quad $u_{\pQ^{hk}}(2;n)$ 
\\[4pt] \hline
Semi-rook & $(1,0)$ & $\displaystyle\frac{n^4}{2} - \frac{n^3}{2}$
\vstrut{22pt} 
\\[10 pt]
Rook & $(2,0)$ & $\displaystyle\frac{n^4}{2} - n^3+\frac{n^2}{2}$
\\[10 pt]
Semibishop & $(0,1)$ & $\displaystyle\frac{n^4}{2} - \frac{n^3}{3}- \frac{n}{6}$ 
\\[10 pt]
Anassa & $(1,1)$ & $\displaystyle\frac{n^4}{2} -\frac{5 n^3 }{6}+ \frac{n^2}{2} - \frac{n}{6}$ 
\\[10 pt]
Semiqueen & $(2,1)$ & $\displaystyle\frac{n^4}{2} -\frac{4 n^3 }{3} + n^2 - \frac{n}{6} $ 
\\[10 pt]
Bishop & $(0,2)$ & $\displaystyle\frac{n^4}{2}-\frac{2 n^3}{3}+\frac{n^2}{2}-\frac{n}{3}$  
\\[10 pt]
Trident & $(1,2)$ & $\displaystyle\frac{n^4}{2} -\frac{7 n^3 }{6} + n^2 - \frac{n}{3}$ 
\\[10 pt]
Queen & $(2,2)$ & $\displaystyle\frac{n^4}{2}-\frac{5 n^3}{3}+\frac{3 n^2}{2}-\frac{n}{3}$ 
\\ [9 pt] \hline
\end{tabular}
\bigskip
\caption{The quasipolynomials that count nonattacking configurations of two partial queens.} 
\label{Tb:pqueenspolys2}
\end{table}
\begin{table}[ht]
\begin{tabular}{|l|l|} \hline
$(h,k)$ \vstrut{15pt} &\qquad $u_{\pQ^{hk}}(3;n)$ 
\\[4pt] \hline
$(1,0)$ 
\vstrut{22pt} 
& $\displaystyle\frac{n^6}{6}-\frac{n^5}{2}+\frac{n^4}{3}$
\\[10 pt]
$(2,0)$ & $\displaystyle\frac{n^6}{6}-n^5+\frac{13 n^4}{6}-2 n^3+\frac{2 n^2}{3}$
\\[10 pt]
$(0,1)$ & $\displaystyle\frac{n^6}{6}-\frac{n^5}{3}+\frac{n^4}{6}-\frac{n^3}{6}+\frac{n^2}{6}$ 
\\[10 pt]
$(1,1)$ & $\displaystyle\frac{n^6}{6}-\frac{5 n^5}{6}+\frac{5 n^4}{3}-\frac{11 n^3}{6}+\frac{7 n^2}{6}-\frac{n}{3}$ 
\\[10 pt]
$(2,1)$ & $\displaystyle\frac{n^6}{6}-\frac{4 n^5}{3}+\frac{25 n^4}{6}-\frac{37 n^3}{6}+\frac{25 n^2}{6}-n$ 
\\[10 pt]
$(0,2)$ & $\displaystyle\frac{n^6}{6}-\frac{2 n^5}{3}+\frac{5 n^4}{4}-\frac{5 n^3}{3}+\frac{4 n^2}{3}-\frac{2 n}{3}+\frac{1}{8}-(-1)^n\frac{1}{8}$  
\\[10 pt]
$(1,2)$ & $\displaystyle\frac{n^6}{6}-\frac{7 n^5}{6}+\frac{41 n^4}{12}-\frac{65 n^3}{12}+\frac{14 n^2}{3}-\frac{43 n}{24}+\frac{1}{8}+(-1)^n\Big\{\frac{n}{8}-\frac{1}{8}\Big\}$ 
\\[10 pt]
$(2,2)$ & $\displaystyle\frac{n^6}{6}-\frac{5 n^5}{3}+\frac{79 n^4}{12}-\frac{25 n^3}{2}+11 n^2-\frac{43 n}{12}+\frac{1}{8}+(-1)^n\Big\{\frac{n}{4}-\frac{1}{8}\Big\}$ 
\\ [9 pt] \hline
\end{tabular}
\bigskip
\caption{The quasipolynomials that count nonattacking configurations of three partial queens.} 
\label{Tb:pqueenspolys3}
\end{table}

In all instances, these equations agree with \Kot's conjectures and data.  
(After we suggested partial queens, \Kot\ computed many values of the counting functions and inferred formulas which we employed to correct and verify our theoretical calculations.)

\begin{proof}
The only subspaces that contribute to $o_{\pQ^{hk}}(3;n) = 3! u_{\pQ^{hk}}(3;n)$ are those that involve three pieces or fewer.  
The subspace $\bbR^{2q}$ of codimension 0 contributes $n^{2q}$.  
The contribution from codimension 1 is given in Equation~\eqref{E:codim1}.  
In the proof of Theorem~\ref{P:hvdiag}, we already have calculated the contributions from subspaces of types $\cU_2^2$,  $\cU_{3\mathrm{a}}^2$, $\cU_{3\mathrm{b}}^2$, $\cU_{3\mathrm{a}}^3$, and $\cU_{3\mathrm{b}}^3$.  There is one final type of subspace, involving three pieces.
\medskip

\begin{description}

\item[{\bf Type $\cU_3^4$}\,]  
  The subspace $\cU$ is defined by four move equations on three pieces that specify that the pieces all occupy one position on the board; that is, $\cU = \cW_{ijl}^{\,=}$.  
 
There is one subspace for each of the $\binom{q}{3}$ unordered triples of pieces.  The number of points in a subspace is $n^2$, the size of the board.
 
According to Lemma~\LWmu, $\mu(\hat0,\cU) = (h+k-1)^2(h+k+2),$ which happily gives 0 when $h+k=1$.
 
Consequently, the contribution to $o_{\pQ^{hk}}(q;n)$ is 
$$\binom{q}{3}(h+k-1)^2(h+k+2) n^{2q-4}.$$

\end{description}

Combining all contributions and dividing by $q!=6$ gives the formula of the theorem.
\end{proof}

We can now calculate the number of combinatorial types for two and three partial queens.  

\begin{cor}\label{C:types23}
The number of combinatorial types of nonattacking configuration of $q$ partial queens $Q^{hk}$ is $h+k$ when $q=2$ and when $q=3$ is given by Table~\ref{Tb:types3}.
\end{cor}
\begin{center}
\begin{table}[ht]
\begin{tabular}{|l||c|c|c|}
\hline
\vstrut{13pt} $h \ \bs\ k$  &\makebox[2em][c]{0}  &\makebox[2em][c]{1}  &\makebox[2em][c]{2}  \\
\hline\hline
\vstrut{12pt}\ 0	&--	&1	&6	\\
\hline
\vstrut{12pt}\ 1	&1	&6	&17	\\
\hline
\vstrut{12pt}\ 2	&6	&17	&36	\\
\hline
\end{tabular}
\bigskip
\caption{The number of combinatorial types of nonattacking configuration for three partial queens.}
\label{Tb:types3}
\end{table}
\end{center}

\begin{proof}
Set $n=-1$ in $u_{Q^{hk}}(q;n)$ and apply Theorem~\Ttypenumber.
\end{proof}

For $q=2$ we get the number of basic moves, in accord with Proposition~\Ptwopiecetypes.  For $q=3$ the number of types depends only on the number of moves, just as when we compared three queens to three nightriders in the end of Section~\types.  The numbers match \cite[Sequence A084990]{OEIS}, whose formula is $s(s^2+3s-1)/3$ with $s:=|\M|$.

\begin{conj}\label{Cj:3piecetypes}
The number of combinatorial configuration types of three pieces is $$|\M|\big(|\M|^2+3|\M|-1\big)/3.$$
\end{conj}

\sectionpage\section{Volumes and Evaluations}\label{last}

Here are an observation and a related problem suggested by our calculation of partial queen coefficients and similar computations for the nightrider in Part~IV.

\subsection{A quasipolynomial observation}\label{numobs}\

It is striking that with Theorems~\ref{P:hvdiag} and \thmN\ we can know the period and periodic part of a quasipolynomial coefficient without knowing anything about the rest of the coefficient.

\subsection{A problem of volumes}\label{vols}

It should be possible to find the volume of $\cU\cap[0,1]^{2q}$ without the trouble of finding its complete Ehrhart quasipolynomial.  Doing so would provide the leading term of $\alpha(\cU;n)$ and thereby the exact contribution of $\cU$ to $\gamma_{\codim\cU}$.  This would be helpful for all pieces, not only partial queens.

The advantage would be that, if $\alpha(\cU;n)$ were known for all subspaces of lesser codimension than $i$ and if $\vol(\cU\cap[0,1]^{2q})$ were known for all subspaces of codimension $i$, then $\gamma_i$ would be completely known.  
Thus we could complete the evaluations of $\gamma_5$ and $\gamma_6$ in Theorem~\ref{P:hvdiag} and of $\gamma_3$ for nightriders in Part~V.

\sectionpage\section{The Missing Proofs}\label{proofs}

We present the proofs of Lemmas~\ref{L:codim012} and \ref{L:codim3}.

\subsection{Proof of Lemma~\ref{L:codim012} on codimension up to $2$}\

\begin{proof}
The case $\nu=0$ is from $\alpha(\ocube;n) = n^{2q}$.  The case $\nu=1$ is that of hyperplanes:

\begin{description}
\item[{\bf Type $\cU_2^1$}\,]
The hyperplanes contribute 
$$
-\binom{q}{2}\sum_{(c,d)\in\M} \alpha^{d/c}(n) \cdot n^{2q-4} = -\binom{q}{2} \Big[ \frac{3h+2k}{3}n^{2q-1} + \frac{k}{3}n^{2q-3} \Big]
$$ 
to $o_{\pQ^{hk}}(q;n)$ since we choose an unordered pair of pieces and a single slope, and the M\"obius function is $-1$, and the number of ways to place the two pieces is $\alpha^{d/c}(n)$, given in Equation~\eqref{E:attacklines}.

\end{description}
\medskip

It remains to solve $\nu=2$.  
We break the subspaces down into four types.  

\begin{description}
\item[{\bf Type $\cU_2^2$}\,]  
The subspace $\cU$ is defined by two move equations involving the same two pieces, $\cU=\cH^{d/c}_{ij} \cap \cH^{d'/c'}_{ij}$ where $d/c\ne d'/c'$ and $i<j$.  Thus, $\cU = \cW_{ij}^{\,=}$,\label{d:W=} the subspace corresponding to the equation $z_i=z_j$, i.e., to two pieces in the same location.

There is one such subspace for each of the $\binom{q}{2}$ unordered pairs of pieces.
There are $n^2$ ways to place the two attacking pieces in $\cU$.
The M\"obius function is $\mu(\hat0,\cU) = h+k-1$, by Lemma~\LWmu.  

The total contribution to $o_{\pQ^{hk}}(q;n)$ is 
$$\binom{q}{2}(h+k-1)n^{2q-2}.$$

\item[{\bf Type $\cU_{3\mathrm{a}}^2$}\,]  
The subspace $\cU$ is defined by two move equations of the same slope involving three pieces, say $\cU=\cH^{d/c}_{12} \cap \cH^{d/c}_{23}$.  This subspace is $\cW_{123}^{\,d/c} = \cH^{d/c}_{12} \cap \cH^{d/c}_{13} \cap \cH^{d/c}_{23}$.  There is one such subspace for each of the $\binom{q}{3}$ unordered triples of pieces.  The number of ways to place the three pieces is $\beta^{d/c}(n)$ in Equation~\eqref{E:attacklines}.  Summing over $(c,d)\in\M$ gives $\big[h+\frac{1}{2}k\big]n^4+\frac{1}{2}kn^2.$

The M\"obius function is $\mu(\hat0,\cU)  = 2$ by Lemma~\LWmu.
The total contribution of this type is 
$$
\binom{q}{3}\big\{(2h+k)n^{2q-2}+kn^{2q-4}\big\}.
$$

\item[{\bf Type $\cU_{3\mathrm{b}}^2$}\,]  
The subspace $\cU$ is defined by two move equations of different slopes involving three pieces, say $\cU=\cH^{d/c}_{12} \cap \cH^{d'/c'}_{23}$.

First, we count the number of ways in which we can place three pieces ($\pP_1$, $\pP_2$, and $\pP_3$) so that $\pP_1$ and $\pP_2$ are on a line of slope $d/c$ and $\pP_2$ and $\pP_3$ are on a line of slope $d'/c'$.
Depending on $d/c$ and $d'/c'$, we have the following numbers of choices for the placements of the chosen pieces in the given attacking configuration:

{\bf Case VH.}
If $\{d/c,d'/c'\}=\{0/1,1/0\}$, we have $n^2$ choices for $\pP_2$; then we place $\pP_1$ in one of $n$ positions in the same column as $\pP_2$ and place $\pP_3$ in one of $n$ positions in the same row as $\pP_2$.  This gives a total of $n^4$ placements of the three pieces.  This case contributes only when $h=2$.

{\bf Case DV.}
If one slope is diagonal and the other vertical or horizontal, we first choose the positions of $\pP_1$ and $\pP_2$, which we specify are attacking each other diagonally.  This can be done in $\alpha^{1/1}(n)$ ways.  Then we place $\pP_3$ in line with $\pP_2$ in $n$ ways.  This gives a total of $\frac{2}{3}n^4+\frac{1}{3}n^2$ placements of the three pieces, contributing $hk$ times.

{\bf Case DD.}
If $\{d/c,d'/c'\}=\{1/1,-1/1\}$, then the number of possibilities for placing $\pP_1$ on the diagonal of slope $+1$ and $\pP_3$ on the diagonal of slope $-1$ depends on the position $(x,y)$ where we place $\pP_2$.  Consider the positions $(x,y)$ satisfying $x\geq y$ and $x+y \leq n$; if we rotate this triangle of positions about the center of the square, we see that there are four points with the same number of possibilities for each position in the triangle, except when $n$ is odd, in which case we must consider the position $(\frac{n+1}{2},\frac{n+1}{2})$ independently.  (See Figure~\ref{fig:square}.)

\begin{figure}[hbt]
\includegraphics[height=1.5in]{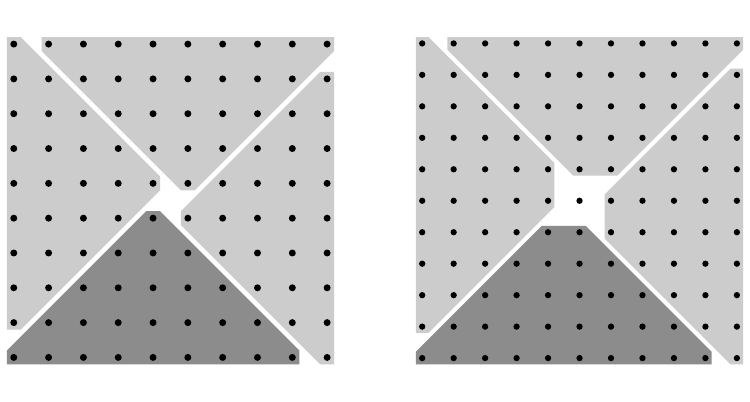}
\caption{The triangle of positions that we consider in Case DD, along with its rotations.  The left figure shows that all positions are covered when $n$ is even; the right figure shows that position $(\frac{n+1}{2},\frac{n+1}{2})$ is considered independently.}
\label{fig:square}
\end{figure}

For a position $(x,y)$ of $\pP_2$ in this triangle, the number of choices for $\pP_1$ is $n-x+y$ and the number of choices for $\pP_3$ is $x+y-1$.   This gives the following number of placements in Case DD:
\pagebreak[2]
\begin{align*}
& \begin{cases} 
\displaystyle 4\sum_{y=1}^{n/2}\sum_{x=y}^{n-y} (n-x+y)(x+y-1) & \text{if $n$ is even,} \\
\displaystyle n^2+4\sum_{y=1}^{(n-1)/2}\sum_{x=y}^{n-y} (n-x+y)(x+y-1) & \text{if $n$ is odd,} 
\end{cases} \\
&= \begin{cases} 
\frac{5}{12}n^4+\frac{1}{3}n^2 & \text{if $n$ is even,} \\[6pt]
\frac{5}{12}n^4+\frac{1}{3}n^2+\frac{1}{4} & \text{if $n$ is odd,} 
\end{cases} \\[6pt]
&= \Big[ \frac{5}{12}n^4+\frac{1}{3}n^2+\frac{1}{8} \Big] - (-1)^n\frac{1}{8}.
\end{align*}
This quantity contributes only when $k=2$. 

In Type $\cU_{3\mathrm{b}}^2$, $\mu(\hat0,\cU)  = 1$.  There are $(q)_3$ ways to choose the three pieces.  The total contribution to $o_{\pQ^{hk}}(q;n)$ depends on $h$ and $k$; it is
\begin{align*}
\qquad\ 
&(q)_3 \bigg\{ \Big[ \delta_{h2} + \frac23 hk + \frac{5}{12} \delta_{k2} \Big] n^{2q-2} + \Big[ \frac13 hk + \frac13 \delta_{k2} \Big] n^{2q-4} + \frac18 \delta_{k2}n^{2q-6} - (-1)^n \frac18 \delta_{k2} n^{2q-6} \bigg\} . 
\end{align*}
\medskip

\item[{\bf Type $\cU_{4^*}^2{:}\cU_2^1\cU_2^1$}\,]   
The subspace $\cU$ is defined by two move equations involving four distinct pieces.  Hence, $\cU = \cH^{d/c}_{12} \cap \cH^{d'/c'}_{34}$, which decomposes into the two hyperplanes $\cH^{d/c}_{12}$ and $\cH^{d'/c'}_{34}$, where $d'/c'$ may equal $d/c$.  
The M\"obius function is $\mu(\hat0,\cU) = 1$.  

There are $2!\binom{q}{2,2,q-4} = (q)_4/4$ ways to choose an ordered pair of unordered pairs of pieces.  Assign any slope $d/c$ to the first pair and $d'/c'$ to the second.  Each pair of slopes, distinct or equal, appears twice, once for each ordering of the unordered pairs, so we divide by 2.  
The number of attacking configurations in each case is $\alpha^{d/c}(n) \cdot \alpha^{d'/c'}(n)$.   
The total contribution of all cases (before multiplication by $n^{2q-8}$) is
\begin{align*}
\qquad
\frac{(q)_4}{8} \sum_{(c,d),(c',d')\in\M} \alpha^{d/c}(n)\cdot\alpha^{d'/c'}(n) 
= \frac{(q)_4}{8} \bigg[ \sum_{(c,d)\in\M} \alpha^{d/c}(n) \bigg]^2 
= \frac{(q)_4}{8} \bigg[ \frac{3h+2k}{3}n^3 + \frac{k}{3}n \bigg]^2 .
\end{align*}
Thus, the contribution of Type $\cU_{4^*}^2$ to $o_{\pQ^{hk}}(q;n)$, after multiplication by the $n^{2q-8}$ ways to place the remaining pieces, is 
\[
\frac{1}{8}(q)_4 \bigg\{\Big[h^2+\frac{4}{3}hk+\frac{4}{9}k^2\Big]n^{2q-2}+\Big[\frac{2}{3}hk+\frac{4}{9}k^2\Big]n^{2q-4}+\frac{1}{9}k^2n^{2q-6}\bigg\}.
\]

\end{description}

Adding up the various types gives the total contribution to $o_{\pQ^{hk}}(q;n)$; dividing by $q!$ concludes the proof of Lemma~\ref{L:codim012}.
\end{proof}

\subsection{Proof of Lemma~\ref{L:codim3} on codimension $3$}\

\begin{proof}
The subspaces $\cU$ defined by three move equations may involve three, four, five, or six pieces.  We treat each number of pieces in turn.

\begin{description}
\item[{\bf Type $\cU_{3\mathrm{a}}^3$}\,]  
The subspace $\cU$ is defined by three move equations of distinct slopes involving the same three pieces, say $\cU=\cH^{d/c}_{12} \cap \cH^{d'/c'}_{13} \cap \cH^{d''/c''}_{23}$ where $d/c$, $d'/c'$, and $d''/c''$ are distinct. 
\par There is one subspace $\cU$ for every valid choice of three slopes and each of the $(q)_3$ ways to choose three pieces and assign pairs of them to the three slopes.
\par As exhibited in Figure~\ref{fig:tri}, there are two kinds of subspace $\cU$, with the hypotenuse of the right triangle either on a diagonal (Case $\triangle_1$) or on a vertical or horizontal line (Case $\triangle_2$).  

\begin{figure}
\includegraphics[height=1.2in]{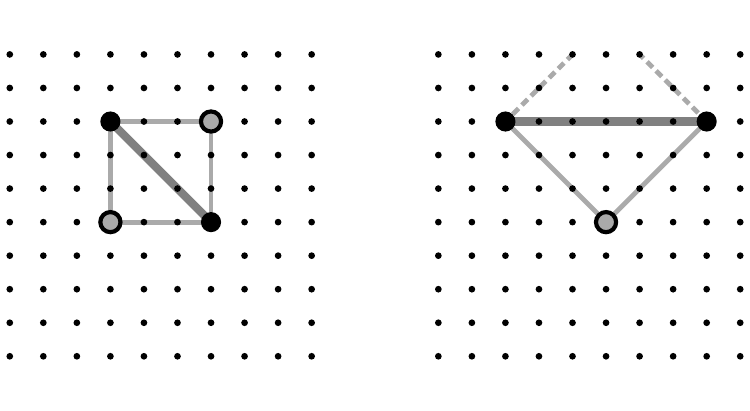}
\caption{Given two attacking queens on the hypotenuse of a right triangle, there may be one or two locations for a third mutually attacking queen, as explained in Type $\cU_{3\mathrm{a}}^3$.}
\label{fig:tri}
\end{figure}

{\bf Case $\triangle_1$.}  
We take $d/c$ to be diagonal with slope $-1/1$ and we take $d'/c'=0/1$ and $d''/c''=1/0$.  That is the upper triangle in Figure~\ref{fig:tri}(left) if $\pP_1$ is higher than $\pP_2$ and the lower one if lower; when $\pP_1$ and $\pP_2$ coincide, the triangle degenerates to a point and $\pP_3$ coincides with the other pieces.

Once we have chosen the positions of $\pP_1$ and $\pP_2$ on a diagonal, $\pP_3$ is determined.  The number of such configurations is $\alpha^{1/1}(n) = \frac{2n^3+n}{3} = \alpha(\cU;n)$.  This is then multiplied by $k$ for the $k$ diagonal slopes $d/c$ and $\delta_{h2}$ because this case exists only when $h=2$.

{\bf Case $\triangle_2$.}  
There are $h$ horizontal and vertical moves, so $h$ orientations for the hypotenuse.  We take the case of a horizontal hypotenuse, $d/c=0/1$, and $d'/c'=1/1$, $d''/c''=-1/1$; then we multiply the count by $h$.  $\pP_3$ is in the upper triangle of Figure~\ref{fig:tri}(right) if $\pP_1$ is left of $\pP_2$ and the lower one if right; when $\pP_1$ and $\pP_2$ coincide, the triangle is degenerate and $\pP_3$ coincides with them.

First we choose the vertical coordinate $y$ of the hypotenuse.  Form the diamond composed of two triangles with vertices $(1,y), (n,y)$, and $(\frac{n+1}{2},y+\frac{n}{2})$ (the upper triangle) or $(\frac{n+1}{2},y-\frac{n}{2})$ (the lower triangle).  $\pP_3$ may have any (integral) location in these triangles that is in the board $[1,n]^2$, and once it is positioned the locations of $\pP_1$ and $\pP_2$ are determined.  Thus, we need only count the valid locations for $\pP_3$ for each height $y$.  We do so by counting the integral points in the diamond and subtracting those outside the board.

The number of integral points in one triangle with hypotenuse of $n$ points is 
$
T(n) := \frac14(n^2+2n+\epsilon), 
$
where $\epsilon := \frac12 [1-(-1)^n] = (n \mod 2)$.\label{d:ep} The number in the diamond is $T(n)+T(n-2) = \frac{1}{2}(n^2+\epsilon)$.  There are $n$ such diamonds.

The number of diamond points outside the board depends on $y$.  For $y = \frac{n+1}{2}$ (the midline, which exists when $n$ is odd), there are no such points.  Thus, we total the count for all $y<n/2$ and double it.  The excluded part of the diamond is a triangle whose upper edge extends from $(y+1,0)$ to $(n-y,0)$, with $n-2y$ points, so the number of excluded points is $T(n-2y) = \frac14[(n-2y+1)^2-1+\epsilon]$.  Summing over $y$ and doubling to include $y>(n+1)/2$ gives 
$
2 \sum_{y=1}^{(n-\epsilon)/2} T(n-2y) = \frac{1}{12}(n^3-4n+3n\epsilon).
$
This is subtracted from the total of the $n$ diamond areas and the result multiplied by $h$, giving 
\begin{align*}
&h \bigg\{ n \frac{n^2+\epsilon}{2} - \frac{n^3-4n+3n\epsilon}{12} \bigg\} = \frac{5h}{12}n^3+\frac{11h}{24}n-(-1)^n\frac{h}{8}n
\end{align*}
as the number of configurations.  This case applies only when $k=2$.

In both cases of type $\cU_{3\mathrm{a}}^3$, $\mu(\hat0,\cU)  =-1$ because the number of hyperplanes that contain $\cU$ is $3 = \codim\cU$.  In each case we multiply by $(q)_3$ for the number of subspaces and $n^{2q-6}$ for the $q-3$ other pieces.  
The total contribution of this type to $o_{\pQ^{hk}}(q;n)$ is 
\begin{align*}
\qquad\quad
&-(q)_3 \bigg\{ \Big[\delta_{h2}\frac{2k}{3} + \delta_{k2}\frac{5h}{12}\Big] n^{2q-3}
 + \Big[\delta_{h2}\frac{k}{3} + \delta_{k2}\frac{11h}{24}\Big] n^{2q-5} - (-1)^n \delta_{k2}\frac{h}{8}n^{2q-5} \bigg\}.
\end{align*}
%

\medskip
\item[{\bf Type $\cU_{3\mathrm{b}}^3$}\,]  
The subspace $\cU$ is defined by three move equations involving three pieces and two or three slopes, of the form $\cU=\cH^{d/c}_{12} \cap \cH^{d'/c'}_{12} \cap \cH^{d''/c''}_{23}$ where $d''/c''$ is any chosen slope, and $d/c, d'/c'$ are arbitrary distinct slopes.  This subspace equals $\cW_{12}^{\,=}\cap\cW_{123}^{d''/c''}$; thus, it does not depend on the choice of $d/c$ and $d'/c'$, and $\cH^{d''/c''}_{23}$ can be replaced by $\cH^{d''/c''}_{13}$ in the definition of $\cU$.  Moreover, the number of ways to place the three pieces equals the number of ways to place an ordered pair of pieces in a line of slope $d''/c''$, i.e., $\alpha^{d''/c''}(n)$ from Equation~\eqref{E:attacklines}.  This should be multiplied by $n^{2q-6}$ for the remaining $q-3$ pieces.

By Lemma~\LWmu\ the M\"obius function is $\mu(\hat0,\cU)=-2(h+k-1)$.  We can specify the pieces involved in $(q)_3/2$ ways.  The contribution to $o_{\pQ^{hk}}(q;n)$ is therefore
\[
-(q)_3(h+k-1)\bigg\{\Big[h+\frac{2k}{3}\Big]n^{2q-3}+\frac{k}{3}n^{2q-5}\bigg\}.
\]

\item[{\bf Type $\cU_{4\mathrm{a}}^3$}\,]   
The subspace $\cU$ is defined by three move equations of the same slope involving four pieces, say $\cU = \cW_{1234}^{\,d/c} =$ (for instance) $\cH^{d/c}_{12} \cap \cH^{d/c}_{23} \cap \cH^{d/c}_{34}$.  There are $\binom{q}{4}$ ways to choose the four pieces.

The number of ways to place four attacking pieces in $\cU$ is $\sum_{l\in\bL^{d/c}(n)} l^4$\label{d:linesizes} (see Section~\oneortwo), which depends on $d/c$.
When $d/c\in\{0/1,1/0\}$, the number is $\sum_{l\in\bL^{d/c}(n)}l^4=n^5$.  
When $d/c\in\{1/1,-1/1\}$, the number is $\sum_{l=1}^n l^4 + \sum_{l=1}^{n-1} l^4 = \frac{1}{15}(6n^5+10n^3-n)$. 

We have $\mu(\hat0,\cU)  = -6$ because $\cU$ is contained in six hyperplanes $\cH^{d/c}_{ij}$, four codimension-2 subspaces of type $\cU_{3\mathrm{a}}^2$, and three codimension-2 subspaces of type $\cU_{4^*}^2$.  The total contribution to $o_{\pQ^{hk}}(q;n)$ is 
\[ 
-\binom{q}{4}\bigg\{\Big[6h+\frac{12k}{5}\Big]n^{2q-3}+4kn^{2q-5}-\frac{2k}{5}n^{2q-7}\bigg\}.
\]

\item[{\bf Type $\cU_{4\mathrm{b}}^3$}\,]  
The subspace $\cU$ is defined by three move equations involving four pieces, two of the equations having the same slope and involving the same piece: say $\cU = \cW_{123}^{\,d/c} \cap \cH^{d'/c'}_{34} =$ (for example) $ \cH^{d/c}_{12} \cap \cH^{d/c}_{23} \cap \cH^{d'/c'}_{34}$, where $d'/c' \neq d/c$. 

There is a subspace for each of $(q)_4/2!$ choices of pieces (since $\pP_1$ and $\pP_2$ are unordered) and for each ordered pair of slopes $d/c$ and $d'/c'$.

Just as with subspaces of type $\cU_{3\mathrm{b}}^2$, we have three cases.

\medskip
{\bf Case VH.}  Take $d/c=1/0$ and $d'/c'=0/1$.  Choosing $\pP_3$'s position in $n^2$ ways, place $\pP_1$ and $\pP_2$ in the same column in $n^2$ ways, and place $\pP_4$ in $\pP_3$'s row in $n$ ways.  Multiply by two for interchanging slopes, for a total of $2n^5$ placements when $h=2$.

\medskip
{\bf Case DV.}  Take $d/c=1/1$ (diagonal) and $d'/c'=1/0$ (vertical).  As in Type $\cU^2_{3\mathrm{a}}$, the number of ways to place $\pP_1$, $\pP_2$, and $\pP_3$ in the same diagonal is given by Equation~\eqref{E:attacklines}; multiply by the $n$ ways to place $\pP_4$ in the same column as $\pP_3$.  Considering the choice of diagonal and that of column or row, we get $\frac{1}{2}hk(n^5+n^3)$.  

Or, take $d/c=1/0$ and $d'/c'=1/1$.  Place $\pP_3$ and $\pP_4$ in the same diagonal in $\alpha^{1/1}(n)$ ways and multiply by $n^2$ placements of $\pP_1$ and $\pP_2$ in $\pP_3$'s column; we get $\frac{1}{3}hk(2n^5+n^3)$.  

The total is 
$hk(\frac{7}{6}n^5+\frac{5}{6}n^3).$

\medskip
{\bf Case DD.}  
This case exists only when $k=2$.  
Here $\{d/c,d'/c'\}=\{1/1,-1/1\}$.  We reduce the computation by symmetry, as in Type $\cU_{3\mathrm{b}}^2$, but here the symmetry in Figure~\ref{fig:square} is broken by having two pieces in one of the diagonals.  Thus, we count the placements where $\pP_3$ is on one of the two main diagonals separately from the other placements.  See Figure~\ref{fig:square3} for a visual representation.

For $\pP_3$ at a point $(x,y)$ in the bottom triangle $y+1 \leq x \leq n-y$, the number of placements with $\pP_1$ and $\pP_2$ on the diagonal of slope $+1$ and $\pP_4$ on the diagonal of slope $-1$ through $(x,y)$ equals the number with $\pP_1$ and $\pP_2$ on the diagonal of slope $-1$ and $\pP_4$ on the diagonal of slope $+1$ through $(n+1-x,y)$, which is also in the bottom triangle. Therefore, if we double the number of the former kind we get the total number with $\pP_3$ in the bottom triangle.  (Note that we are combining the counts of two different [but isomorphic] subspaces $\cU$.  In particular, the configuration with all pieces in the same place is counted twice, but only once for each subspace.)  
Multiplying this by 4 for the four triangles, we have the number of configurations where $\pP_3$ is off the two main diagonals.
To get the actual number note that when $(x,y)$ is in the bottom triangle, its positive diagonal has $n-x-y$ points and its negative diagonal has $x+y-1$ points.

Similarly, if we count the configurations with $\pP_3$ in the lower left or lower right half-diagonal and $\pP_1, \pP_2$ on the diagonal with positive slope, double the result.  We double this again to account for the upper half-diagonals.

When $n$ is odd, the center point contributes $n^3$ for each choice of the diagonal of $\pP_1,\pP_2$.

Thus we have the following number of placements in Case DD:
\begin{align*}\qquad
&\begin{cases} 
\begin{aligned}
 &8\sum_{y=1}^{n/2}\sum_{x=y+1}^{n-y} (n-x+y)^2(x+y-1) \\&\qquad+4\sum_{y=1}^{n/2}[n(2y-1)^2+n^2(2y-1)] 
\end{aligned}
& \text{if $n$ is even,} \\
\begin{aligned}
 &8\sum_{y=1}^{(n-1)/2}\sum_{x=y+1}^{n-y} (n-x+y)^2(x+y-1) \\&\qquad+4\sum_{y=1}^{(n-1)/2}[n(2y-1)^2+n^2(2y-1)]+2n^3 
\end{aligned}
& \text{if $n$ is odd,} 
\end{cases} \\[3pt]
&=\begin{cases} 
\frac{3}{5} n^5 + \frac{2}{3} n^3 - \frac{4}{15} n & \text{if $n$ is even,} \\[3pt]
\frac{3}{5} n^5 + \frac{2}{3} n^3 + \frac{11}{15} n & \text{if $n$ is odd,} 
\end{cases} \\[3pt]
&\quad=\left[\frac{3}{5} n^5 + \frac{2}{3} n^3 + \frac{7}{30} n\right] - (-1)^n \frac{1}{2} n .
\end{align*}

\begin{figure}[hbt]
\includegraphics[height=1.5in]{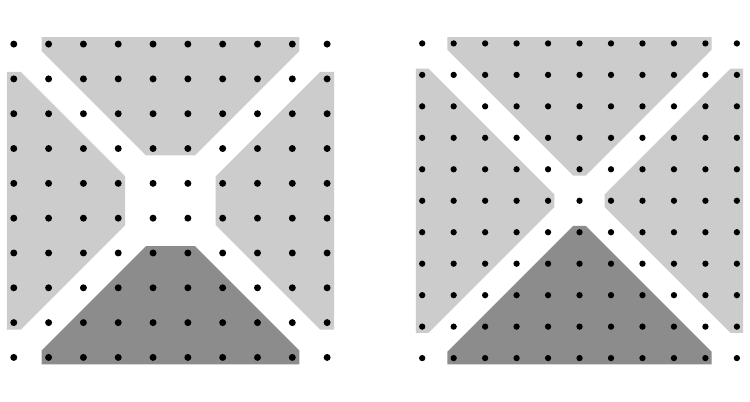}
\caption{The triangle of positions that we consider in Case DD of Type $\cU_{4\mathrm{b}}^3$, for even $n$ (left) and odd $n$ (right).}
\label{fig:square3}
\end{figure}

\medskip
In all cases of Type $\cU_{3\mathrm{b}}^3$, $\mu(\hat0,\cU) = -2$ because $\cU$ is contained in four hyperplanes, $\cH_{ij}^{d/c}$ with $i,j\in\{1,2,3\}$ and $\cH_{34}^{d'/c'}$, and the four subspaces $\cW_{123}^{\,d/c}$ and $\cH_{ij}^{d/c} \cap \cH_{34}^{d'/c'}$ of codimension 2.

Therefore, the total contribution to $o_{\pQ^{hk}}(q;n)$ from Type $\cU_{4\mathrm{b}}^3$ is
\begin{align*}
\qquad\quad 
-(q)_4 &\left\{ \Big[2\delta_{h2}+\frac{3}{5}\delta_{k2}+\frac{7}{6}hk\Big]n^{2q-3} + \Big[\frac23\delta_{k2}+\frac56hk\Big]n^{2q-5} +\frac{7}{30}\delta_{k2}n^{2q-7}  \right. \\
&\quad - \left. (-1)^n \frac{1}{2}\delta_{k2}n^{2q-7} \right\}.
\end{align*}
%

\medskip
\item[{\bf Type $\cU_{4\mathrm{c}}^3$}\,]  
The subspace $\cU$ is defined by three move equations, two having the same slope but not involving the same piece, say $\cU=\cH^{d/c}_{12} \cap \cH^{d'/c'}_{23} \cap \cH^{d/c}_{34}$. 
There are $(q)_4/2$ choices for $\pP_1$ through $\pP_4$ because of the symmetry.
\par We have the following cases (see Figure~\ref{fig:34c}):
\begin{figure}[hbt]
\includegraphics[height=1.2in]{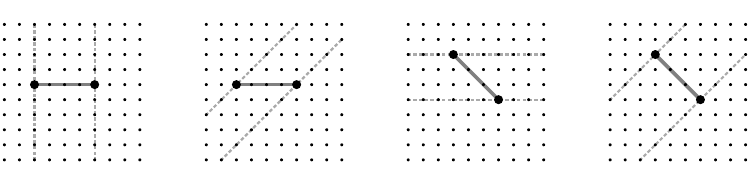}
\caption{The possible attacking configurations in Type $\cU_{4\mathrm{c}}^3$.  From left to right are cases VHV, DHD, HDH, and DDD.}
\label{fig:34c}
\end{figure}
%

\medskip
{\bf Case VHV.}
If $\{d/c,d'/c'\}=\{0/1,1/0\}$, we can choose the pieces and positions for $\pP_2$ and $\pP_3$ in a row in $n^3$ ways, and then place $\pP_1$ in $\pP_2$'s column and $\pP_4$ in $\pP_3$'s column in $n^2$ ways.  With two possible orientations (VHV or HVH), the number of attacking configurations is $2n^5$ when $h=2$. 

\medskip
{\bf Case DHD.}
We consider the case where the outer attacking move is diagonal and the inner attacking move is horizontal or vertical.  Without loss of generality, suppose $d/c=+1/1$ and $d'/c'=0/1$.   We investigate the possibilities for $\pP_1$ and $\pP_4$ based on choosing the row for $\pP_2$ and $\pP_3$.

Suppose that $\pP_2$ and $\pP_3$ are in row $y$, where $1\leq y \leq n$.  The positions that do not diagonally attack a position in row $y$ are those in two right triangles, one in the upper left and the other in the lower right, with legs having, respectively, $n-y$ and $y-1$ points.  Placing $\pP_1$ and $\pP_4$ in any attacking positions determines where $\pP_2$ and $\pP_3$ are.  Thus, the number of configurations is
$\sum_{y=1}^n \big[n^2 - \binom{n-y+1}{2} - \binom{y}{2}\big]^2 = \frac{9}{20}n^5+\frac{5}{12}n^3+\frac{2}{15}n,$
which contributes $hk$ times.

\medskip
{\bf Case HDH.}
When the inner attacking move is diagonal and the outer attacking move is horizontal or vertical, we first choose the positions of $\pP_2$ and $\pP_3$, in one of $\alpha^{1/1}(n)$ ways.  There are $n^2$ ways to place $\pP_1$ in relation to $\pP_2$ and $\pP_4$ in relation to $\pP_3$, giving a total contribution of $hk(\frac{2}{3}n^5+\frac{1}{3}n^3)$.

\medskip
{\bf Case DDD.} 
Here $\{d/c,d'/c'\}=\{1/1,-1/1\}$; say $d/c=1/1$.  
We first determine the number of positions diagonally attacking a piece placed in a diagonal $D_y$ of slope $-1$ passing through $(1,y)$ for a fixed $y \in [2n-1]$.  As $y$ varies, the multiset of the number of positions attacking the positions on it along each opposite diagonal has the following pattern:
\begin{align*}
D_1,D_{2n-1}\text{: } &\{n\}, \\
D_2,D_{2n-2}\text{: } &\{n-1,n-1\}, \\
D_3,D_{2n-3}\text{: } &\{n-2,n,n-2\}, \\
D_4,D_{2n-4}\text{: } &\{n-3,n-1,n-1,n-3\}, \\
	\ldots\qquad &\ \ldots, \\
D_{n-1},D_{n+1}\text{: } &\begin{cases}
	\{3,\ldots,n-1,n-1,\ldots,3\}	&\qquad\text{ if $n$ is even},\\
	\{3,\ldots,n-2,n,n-2,\ldots,3\}	&\qquad\text{ if $n$ is odd},
	\end{cases}\\
D_n\text{: } 	&\begin{cases}
	\{1,3,\ldots,n-1,n-1,\ldots,3,1\}	&\text{ if $n$ is even},\\
	\{1,3,\ldots,n-2,n,n-2,\ldots,3,1\}	&\text{ if $n$ is odd}.
	\end{cases}
\end{align*}
Then $\pP_1$ and $\pP_4$ can each be placed arbitrarily and independently in any of the opposite diagonals that attack $D_y$.  The choice of the opposite diagonal determines the locations of $\pP_2$ and $\pP_3$, respectively.  Given $y$, the number of placements of $\pP_1$ and $\pP_4$ is the square of the sum of all lengths in $D_y$; thus, the total number of ways to place the four pieces is
$$
\qquad\quad
\begin{cases} 
\begin{aligned}
&\textstyle 2\sum _{j=0}^{(n/2)-1} \big[n+2 \sum_{i=1}^j (n-2 i)\big]^2 + 2\sum_{j=0}^{(n/2)-2} \big[2 \sum_{i=0}^j (n-2 i-1)\big]^2 \\[3pt]
&\qquad\textstyle + \big[2 \sum_{i=0}^{(n/2)-1} (n-2 i-1)\big]^2
\end{aligned}
  & \text{if $n$ is even,} \\[20pt]
\begin{aligned}
&\textstyle 2\sum_{j=0}^{(n-3)/2} \big[n+2 \sum_{i=1}^j (n-2 i)\big]^2 + 2\sum_{j=0}^{(n-3)/2} \big[2 \sum_{i=0}^j (n-2 i-1)\big]^2 \\[3pt]
&\qquad\textstyle + \big[n+2 \sum _{i=1}^{(n-1)/2} (n-2 i)\big]^2
\end{aligned}
  & \text{if $n$ is odd,} 
\end{cases}
$$
which simplifies for both parities to $\frac{4}{15}n^5+\frac{1}{3}n^3+\frac{2}{5}n$.  
We double this quantity for the second subspace resulting from choosing slope $d/c=-1$.  
The result is $\frac{8}{15}n^5+\frac{2}{3}n^3+\frac{4}{5}n,$ valid when $k=2$.

In this type, once again, $\mu(\hat0,\cU)  =-1$.  The total contribution to $o_{\pQ^{hk}}(q;n)$ is 
\begin{align*}
\qquad\quad
-(q)_4 \left\{ \Big[ \frac{67}{120}hk+\delta_{h2}+\frac{4}{15}\delta_{k2} \Big]n^{2q-3} 
+ \Big[ \frac{3}{8}hk+\frac{1}{3}\delta_{k2} \Big]n^{2q-5} 
+ \Big[ \frac{1}{15}hk+\frac{2}{5}\delta_{k2} \Big]n^{2q-7} \right\}.
\end{align*}
%

\medskip
\item[{\bf Type $\cU_{4\mathrm{d}}^3$}\,]  
The subspace $\cU$ is defined by three move equations having distinct slopes,  say $\cU=\cH^{d/c}_{12} \cap \cH^{d'/c'}_{23} \cap \cH^{d''/c''}_{34}$.  
The arguments here are similar to those for Type $\cU_{4\mathrm{c}}^3$; however, because of the lack of symmetry, there are now $(q)_4$ choices for the pieces $\pP_1$ through $\pP_4$, provided we fix $d/c$ and $d''/c''$.
\par We place pieces $\pP_2$ and $\pP_3$ first, and then pieces $\pP_1$ and $\pP_4$. 
\par Figure~\ref{fig:34d} shows the four cases we consider.
\begin{figure}[hbt]
\includegraphics[height=1.2in]{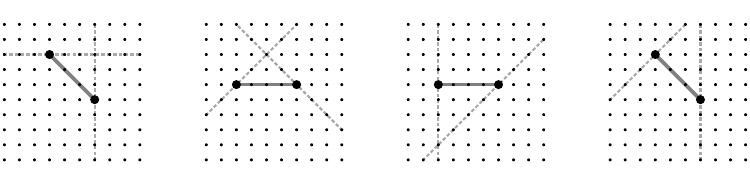}
\caption{The possible attacking configurations in Type $\cU_{4\mathrm{d}}^3$.  From left to right are cases HDV, DHD, VHD, and DDV.}
\label{fig:34d}
\end{figure}
%

\medskip
{\bf Case HDV.}  We assume $d/c=0/1$ and $d''/c''=1/0$.  The argument is the same as in case HDH of Type $\cU_{4\mathrm{c}}^3$.  
The contribution is
$\frac{2}{3}kn^5+\frac{1}{3}kn^3$ when $h=2$.

\medskip
{\bf Case DHD.}  We assume $d/c=1/1$ and $d''/c''=-1/1$.  This case has the same contribution as case DHD of Type $\cU_{4\mathrm{c}}^3$, namely, $\frac{9}{20}hn^5+\frac{5}{12}hn^3+\frac{2}{15}hn$ when $k=2$.

\medskip
{\bf Case VHD.}  We choose $d/c=0/1$ and assume $d''/c''=1/1$.  
We first place $\pP_3$ and $\pP_4$ on the diagonal in $\alpha^{1/1}(n)$ ways, then place $\pP_2$ and $\pP_1$ from $\pP_3$ in $n^2$ ways.  We double for the two orderings of the slopes $0/1$ and $1/0$ and multiply by $h$ for the possible diagonal slopes $d''/c''$.  The contribution here is $\frac{4}{3}kn^5+\frac{2}{3}kn^3$, applicable when $h=2$.

\medskip
{\bf Case DDV.}  We choose $d/c=1/1$ and assume $d''/c''=1/0$.  
Case DD in Type $\cU_{3\mathrm{b}}^2$ counts configurations of $\pP_1$, $\pP_2$, and $\pP_3$ in two attacking moves along diagonals of slopes $+1$ and $-1$.  Then we place $\pP_4$ in relation to $\pP_3$ in $n$ ways.  Accounting for the two different orderings of the slopes $1/1$ and $-1/1$, the contribution when $k=2$ is 
$\big[\frac{5}{6}hn^5+\frac{2}{3}hn^3+\frac{1}{4}hn\big]-(-1)^n\frac{1}{4}hn$.

\medskip
In all cases, $\mu(\hat0,\cU)=-1$.
The total contribution to $o_{\pQ^{hk}}(q;n)$ is 
\begin{align*}
\ \ 
-(q)_4\bigg\{ 
& \Big[2k\delta_{h2}+\frac{77h}{60}\delta_{k2}\Big]n^{2q-3}+\Big[k\delta_{h2}+\frac{13h}{12}\delta_{k2}\Big]n^{2q-5}+\frac{23h}{60}\delta_{k2}n^{2q-7}  
- (-1)^n\frac{h}{4}\delta_{k2}n^{2q-7}\bigg\}.
\end{align*}
%

\medskip
\item[{\bf Type $\cU_{4\mathrm{e}}^3$}\,]   
The subspace $\cU$ is defined by three move equations of different slope, all involving the same piece, say $\cU=\cH^{d/c}_{12} \cap \cH^{d'/c'}_{13} \cap \cH^{d''/c''}_{14}$.  
Given the set of slopes, there are $(q)_4$ ways to choose the pieces.

The number of ways to place four attacking pieces in $\cU$ depends on the slopes.  When $\{1/1,-1/1\}\subset\{d/c,d'/c',d''/c''\}$, then first place $\pP_1$ and the two pieces defined along diagonals as in Case DD from Type $\cU_{3\mathrm{b}}^2$ and subsequently the last piece horizontally or vertically in $n$ ways, giving 
$
hn\left\{\left[\frac{5}{12}n^4+\frac{1}{3}n^2+\frac{1}{8}\right]-(-1)^n\frac{1}{8}\right\}
$
ways for the four pieces (contributing only when $k=2$).  

When $\{0/1,1/0\}\subset\{d/c,d'/c',d''/c''\}$, then place $\pP_1$ and the piece aligned diagonally in $k\alpha^{1/1}$ ways and place the other two pieces in $n^2$ ways, giving $kn^2\left\{ \frac23n^3+\frac13n \right\}$ placements (that contribute only when $h=2$). 

Once more, $\mu(\hat0,\cU)  = -1$.  The total contribution to $o_{\pQ^{hk}}(q;n)$ is
\begin{align*}
\qquad\ 
-(q)_4 \bigg\{ & \Big[\frac{5h}{12}\delta_{k2}+\frac{2k}{3}\delta_{h2}\Big]n^{2q-3} + \Big[\frac{h}{3}\delta_{k2}+\frac{k}{3}\delta_{h2}\Big]n^{2q-5} + \frac{h}{8}\delta_{k2}n^{2q-7}  
 - (-1)^n\frac{h}{8}\delta_{k2}n^{2q-7} \bigg\}.
\end{align*}
%

\medskip
\item[{\bf Type $\cU_{4^*}^3{:}\cU_2^1\cU_2^2$}\,]  
The subspace $\cU$ decomposes into a hyperplane $\cH^{d/c}_{12}$ and a codimension-2 subspace $\cW_{34}^{\,=}$ of type $\cU_2^2$.  We write  $\cW_{34}^{\,=} = \cH^{d'/c'}_{34}\cap \cH^{d''/c''}_{34}$, 
where $d'/c'\ne d''/c''$.  There is no restriction on $d/c$. 

There are $(q)_4/4$ ways to choose the ordered pair of pairs of pieces, $\{\pP_1,\pP_2\}$ and $\{\pP_3,\pP_4\}$.

Since $\pP_4$ is essentially merged with $\pP_3$, the number of attacking configurations is $\sum_{(c,d)\in\M} \alpha^{d/c}(n)=\big(h+\frac{2}{3}k\big)n^3+\frac{1}{3}kn.$ 

The M\"obius function is a product, $\mu(\hat0,\cU)=\mu(\hat0,\cH^{d/c}_{12})\mu(\hat0,\cW_{34}^{\,=}) = 1-|\M|$ (see Section \subspaces).
The total contribution to $o_{\pQ^{hk}}(q;n)$ is 
$$
-(q)_4(h+k-1)\bigg\{\Big[\frac{h}{4}+\frac{k}{6}\Big]n^{2q-3}+\frac{k}{12}n^{2q-5}\bigg\}.
$$ 

\item[{\bf Type $\cU_{5^*\mathrm{a}}^3{:}\cU_2^1\cU_{3\mathrm{a}}^2$}\,]  
The subspace $\cU$ decomposes into a hyperplane and a codimension-2 subspace of type $\cU_{3\mathrm{a}}^2$, say $\cU = \cH^{d/c}_{12} \cap \cW_{345}^{\,d'/c'}$, where $d/c$ may equal $d'/c'$.  
We can choose the pieces in $(q)_5/2!3!$ ways.

The number of attacking configurations is $\sum_{(c,d)\in\M} \alpha^{d/c}(n)=\big(h+\frac{2}{3}k\big)n^3+\frac{1}{3}kn$ times the count from Type $\cU_{3\mathrm{a}}^2$, $\big(h+\frac{1}{2}k\big)n^{4}+\frac{1}{2}kn^{2}$.

As for the M\"obius function, $\mu(\hat0,\cU) = \mu(\hat0,\cH^{d/c}_{12})\mu(\hat0,\cU_{3\mathrm{a}}^2) = -2$.  The contribution to $o_{\pQ^{hk}}(q;n)$ is therefore
\[
-(q)_5\bigg\{ \Big[\frac{h^2}{6}+\frac{7hk}{36}+\frac{k^2}{18}\Big]n^{2q-3}+\Big[\frac{5hk}{36}+\frac{k^2}{12}\Big]n^{2q-5}+\frac{k^2}{36}n^{2q-7}\bigg\}.
\]

\medskip
\item[{\bf Type $\cU_{5^*\mathrm{b}}^3{:}\cU_2^1\cU_{3\mathrm{b}}^2$}\,]  
The subspace $\cU$ decomposes into a hyperplane and a codimension-2 subspace of type $\cU_{3\mathrm{b}}^2$.  We write  $\cU = \cH^{d/c}_{12} \cap \cH^{d'/c'}_{34}\cap \cH^{d''/c''}_{45}$, with $d'/c'\ne d''/c''$ and arbitrary $d/c$.
We can choose the five pieces in $(q)_5/2!$ ways.

The number of attacking configurations is $\sum_{(c,d)\in\M} \alpha^{d/c}(n)$ times the count from Type $\cU_{3\mathrm{b}}^2$, thus
$
\big[ \delta_{h2} + \frac23 hk + \frac{5}{12} \delta_{k2} \big] n^4 + \big[ \frac13 hk + \frac13 \delta_{k2} \big] n^2 + \frac18 \delta_{k2} - (-1)^n \frac18 \delta_{k2} .
$

Here again $\mu(\hat0,\cU) = -1$.  
Consequently, the total contribution to $o_{\pQ^{hk}}(q;n)$ is 
%
\begin{align*}
-\frac12 (q)_5\bigg\{ 
&\Big[ \frac{3h+2k}{3}\delta_{h2} + \frac{2}{3}h^2k + \frac{4}{9}hk^2 + \frac{5(3h+2k)}{36} \delta_{k2} \Big] n^{2q-3} \\
&+\Big[ \frac{1}{3}k\delta_{h2} + \frac{1}{3}h^2k + \frac{4}{9}hk^2 + \Big(\frac{1}{3}h+\frac{13}{36}k\Big)\delta_{k2} \Big] n^{2q-5} \\
&+\Big[ \frac{1}{9}hk^2 + \Big(\frac{1}{8}h+\frac{7}{36}k\Big)\delta_{k2} \Big] n^{2q-7} 
+ \frac{1}{24}k\delta_{k2} n^{2q-9} \\
&- (-1)^n \left( \frac{3h+2k}{24} \delta_{k2} n^{2q-7} + \frac{1}{24}k\delta_{k2} n^{2q-9} \right)
\bigg\}.
\end{align*}
%

\medskip
\item[{\bf Type $\cU_{6^*}^3{:}\cU_2^1\cU_2^1\cU_2^1$}\,]   
The subspace $\cU$ is defined by three move equations involving six distinct pieces.  Thus, $\cU = \cH^{d/c}_{12} \cap \cH^{d'/c'}_{34}\cap \cH^{d''/c''}_{56}$ is decomposable into the three indicated hyperplanes, whose slopes are not necessarily distinct.  The M\"obius function is $\mu(\hat0,\cU) = -1$.  

There are $\binom{q}{2,2,2,q-6} = (q)_6/48$ ways to choose an unordered triple of unordered pairs of pieces.  Then we fix an arbitrary ordering of the three pairs and assign any slope $d/c$ to the first pair, $d'/c'$ to the second, and $d''/c''$ to the third.  
The number of attacking configurations in each case is $\alpha^{d/c}(n) \cdot \alpha^{d'/c'}(n) \cdot \alpha^{d''/c''}(n)$.   
The total contribution of all cases (before multiplication by $n^{2q-12}$) is
\begin{align*}
&-\frac{(q)_6}{48} \sum_{(c,d),(c',d'),(c'',d'')\in\M} \alpha^{d/c}(n)\cdot\alpha^{d'/c'}(n)\cdot\alpha^{d''/c''}(n) \\
&= -\frac{(q)_6}{48} \bigg[ \sum_{(c,d)\in\M} \alpha^{d/c}(n) \bigg]^3 
= -\frac{(q)_6}{48} \bigg[ \frac{3h+2k}{3}n^3 + \frac{k}{3}n \bigg]^3 .
\end{align*}
Thus, the contribution of Type $\cU_{6^*}^3$ to $o_{\pQ^{hk}}(q;n)$, after multiplication by the $n^{2q-12}$ ways to place the remaining pieces, is 
\begin{align*}
\qquad\quad 
&-(q)_6 \left\{ \frac{(3h+2k)^3}{1296} n^{2q-3} + \frac{3k(3h+2k)^2}{1296} n^{2q-5} + \frac{3k^2(3h+2k)}{1296} n^{2q-7} + \frac{k^3}{1296} n^{2q-9}  \right\}.
\end{align*}
\end{description}

Summing the contributions of each type completes the proof of Lemma~\ref{L:codim3}.  (We verified the sum via Mathematica.)
\end{proof}

\sectionpage\section*{Acknowledgement}

We wish to reiterate our gratitude to Vaclav \Kot\ for collecting and computing so many formulas for non-attacking rider configurations.  Besides noticing patterns in the formulas, we relied upon his work to detect and help correct several minor errors that could have spoiled our results.

\newpage\section*{Dictionary of Notation}

\hspace{-.29in}
\begin{minipage}[t]{3.4in}\vspace{0in}
\begin{enumerate}[]
\item $(c,d),(c_r,d_r)$ -- coords of move vector (p.\ \pageref{d:mr})
\item $(\hatc,\hatd)$ -- $(\min,\max)$ of $c,d$ (p.\ \pageref{d:cdhat})
\item $d/c$ -- slope of line or move (p.\ \pageref{d:slope-hyp})
\item $h$ -- \# horiz, vert moves of partial queen (p.\ \pageref{d:partQ})
\item $k$ -- \# diagonal moves of partial queen (p.\ \pageref{d:partQ})
\item $m = (c,d)$, $m_r = (c_r,d_r)$ -- basic move (p.\ \pageref{d:mr})
\item $n$  -- size of square board (p.\ \pageref{d:n})
\item $n+1$ -- dilation factor for board (p.\ \pageref{d:n+1}) 
\item $[n] = \{1,\hdots,n\}$  (p.\ \pageref{d:[n]}) 
\item $[n]^2$ -- square board (p.\ \pageref{d:[n]2}) 
\item $o_\pP(q;n)$ -- \# nonattacking lab configs (p.\ \pageref{d:distattacks})
\item $q$ -- \# pieces on a board (p.\ \pageref{d:q})
\item $r$ -- move index (p.\ \pageref{d:mr})
\item $u_\pP(q;n)$ -- \# nonattacking unlab configs (p.\ \pageref{d:indistattacks})
\item $z=(x,y)$, $z_i=(x_i,y_i)$ -- piece position (p.\ \pageref{d:config})
\end{enumerate}
\medskip
\begin{enumerate}[]
\item $\bz = (z_1,\ldots,z_q)$ -- vector in $\bbR^{2q}$
\item $\bz = (z_1,\ldots,z_q)$ -- configuration (p.\ \pageref{d:config})
\end{enumerate}
\medskip
\begin{enumerate}[]
\item $\alpha(\cU;n)$ -- \# attacking configs in $\cU$ (p.\ \pageref{d:alphaU})
\item $\alpha^{d/c}(n)$ -- \# 2-piece collinear attacks (p.\ \pageref{d:adc})
\item $\beta^{d/c}(n)$ -- \# 3-piece collinear attacks (p.\ \pageref{d:adc})
\item $\gamma_i$ -- coefficient of $u_\pP$ (p.\ \pageref{d:gamma})
\item $\delta_{ij}$ -- Kronecker delta (p.\ \pageref{d:KD})
\item $\epsilon = \frac12 [1-(-1)^n] \equiv n \mod 2$ (p.\ \pageref{d:ep})
\item $\nu$ -- $\codim\cU$ (p.\ \pageref{d:codim})
\item $\mu$ -- M\"obius function of $\cL$ (p.\ \pageref{d:mu})
\item $\kappa$ -- \# of pieces in eqns of $\cU$ (p.\ \pageref{d:kappa})
\end{enumerate}
\end{minipage}
\begin{minipage}[t]{3.5in}\vspace{0in}
\begin{enumerate}[]
\item $\M$ -- set of basic moves (p.\ \pageref{d:moveset})
\end{enumerate}
\bigskip
\begin{enumerate}[]
\item $\cA_{\pP}$ -- move arr of piece $\pP$ (p.\ \pageref{d:AP})
\item $\cH_{ij}^{d/c}$ -- hyperplane for move $(c,d)$ (p.\ \pageref{d:slope-hyp})
\item $\cL$ -- intersection semilattice  (p.\ \pageref{d:L})
\item $\cube$ -- polytope (p.\ \pageref{d:cP})
\item $(\cube,\cA_\pP)$ -- inside-out polytope (p.\ \pageref{d:cP})
\item $\cU$ -- subspace in intersection semilatt (p.\ \pageref{d:U})
\item $\cU^\nu_{\kappa\mathrm{a}}$ -- subsp of codim $\nu$ with $\kappa$ moves (p.\ \pageref{d:codim})

\item $\cW_{i\ldots}^{\,d/c}$ -- subspace of collinearity (p.\ \pageref{d:Wdc})
\item $\cW_{i\ldots}^{\,=}$ -- subspace of equal position (p.\ \pageref{d:W=})
\end{enumerate}
\bigskip
\begin{enumerate}[]
\item $\bbR$ -- real numbers
\item $\bbZ$ -- integers
\end{enumerate}
\bigskip
\begin{enumerate}[]
\item $\pP$ -- piece (p.\ \pageref{d:P})
\item $\pP_i$ -- $i$-th labelled copy of $\pP$ (p.\ \pageref{d:P})
\item $\pQ^{hk}$ -- partial queen (p.\ \pageref{d:partQ})
\end{enumerate}
\vspace{1.9in}
\end{minipage}

\newpage


\newcommand\otopu{$\overset{\circ}{\textrm u}$}

\end{document}